\documentclass[a4paper]{amsart}

\usepackage{hyperref}			
\usepackage{amssymb}			
\usepackage{amsmath}			
\usepackage{mathrsfs}			
\usepackage[all]{xy}			
\usepackage{cite}			
\usepackage{enumerate}			

\theoremstyle{definition}
\newtheorem{mydef}{Definition}[section]

\theoremstyle{plain}
\newtheorem{mythm}[mydef]{Theorem}
\newtheorem{myprop}[mydef]{Proposition}

\newtheorem{mycor}[mydef]{Corollary}
\newtheorem{mylem}[mydef]{Lemma}

\theoremstyle{remark}

\newtheorem{myex}[mydef]{Example}

\newcommand{\btimes}[2]{{}_{#1}\!\times_{#2}}				
\newcommand{\id}{\operatorname{id}}					
\DeclareMathOperator*{\colim}{colim}					
\DeclareMathOperator{\sd}{Sd}						
\DeclareMathOperator{\ex}{Ex}
\DeclareMathOperator{\Hom}{Hom}
\newcommand{\osim}{{\mathord{\sim}}}					
\newcommand{\inj}{\ensuremath{\mathrm{-inj}}}
\newcommand{\proj}{\ensuremath{\mathrm{-proj}}}
\newcommand{\fib}{\ensuremath{\mathrm{-fib}}}
\newcommand{\cof}{\ensuremath{\mathrm{-cof}}}

\newcommand{\cat}[1]{\mathcal{#1}}					

\makeatletter
\newcommand*{\xrightrightarrows}[2]{\mathrel{
  \settowidth{\@tempdima}{$\scriptstyle#1$}
  \settowidth{\@tempdimb}{$\scriptstyle#2$}
  \ifdim\@tempdimb>\@tempdima \@tempdima=\@tempdimb\fi
  \mathop{\vcenter{
    \offinterlineskip\ialign{\hbox to\dimexpr\@tempdima+1em{##}\cr
    \rightarrowfill\cr\noalign{\kern.5ex}
    \rightarrowfill\cr}}}\limits^{\!#1}_{\!#2}}}
\makeatother

\newcommand{\adjct}[4]{#1\colon\cat{#2}\leftrightarrows\cat{#3}:\!#4}	

\newcommand{\nerve}{N}							
\newcommand{\real}[1]{\left|#1\right|}					
\newcommand{\thol}{\tau_1\sd^2}						
\newcommand{\thor}{\ex^2\nerve}						

\newcommand{\trisimp}{\xymatrix@C=0pt@M=0pt@R=.5pt{&\cdot\ar@{-}[dr]&\\\cdot\ar@{-}[ur]\ar@{-}[r]&&\cdot}}

\newcommand{\naturalto}{\Rightarrow}

\newcommand{\Ac}{\mathbf{Ac}}						
\newcommand{\sSet}{\mathbf{sSet}}					
\newcommand{\Cat}{\mathbf{Cat}}						

\newcommand{\ie}{\emph{i.e.\ }}
\newcommand{\wrt}{with respect to\ }

\newcommand{\horn}[2]{\Lambda_#2^#1}
\newcommand{\stdsimp}[1]{\Delta^#1}

\begin{document}

\title{A Model structure on the Category of small acyclic Categories}
\author{Roman Bruckner}
\address{Fachbereich Mathematik und Informatik, Universit\"at Bremen,
  Bibliothekstra\ss{}e~1, 28359~Bremen, Bremen, Germany.}
\email[Roman Bruckner]{bruckner@math.uni-bremen.de }

\begin{abstract}
	In this paper, we show that the Thomason model structure restricts to a Quillen equivalent cofibrantly generated model structure on the category of acyclic categories, whose generating cofibrations are the same as those generating the Thomason model structure. To understand the Thomason model structure, we need to have a closer look at the (barycentric) subdivision endofunctor on the category of simplicial sets.
	This functor has a well known right adjoint, called Kan's Ex functor. Taking the subdivision twice and then the fundamental category yields a left adjoint of an adjunction between the category of simplicial sets and the category of small categories, whose right adjoint is given by applying the Ex functor twice on the nerve of a category. This adjunction lifts the cofibrantly generated Quillen model structure on simplicial sets to a cofibrantly generated model structure on the category of small categories, the Thomason model structure. The generating sets are given by the image of the generating sets of the Quillen model structure on simplicial sets under the aforementioned adjunction.
	We furthermore show that the category of acyclic categories is proper and combinatorial with respect to said model structure. That is weak equivalences behave nicely with respect to pushouts along fibrations and cofibrations, and cofibrations satisfy certain smallness conditions which allow us to work with sets instead of proper classes.
\end{abstract}

\maketitle

\pagenumbering{arabic}
\setcounter{page}{1}

\section{Introduction}
An acyclic category is a category without inverses and non-identity endomorphisms. Acyclic categories have been known under several names. They were called small categories without loops, or scowls, by Haefliger in \cite{haefl:npc}, and loop-free categories by Haucourt \cite{haucourt:lfc} and probably several others. In this paper we adapt the terminology from  \cite{kozlov:caltop} and call them acyclic categories.
Aside from the categorical perspective, we can view acyclic categories as generalized posets, allowing more than one morphism between any ordered pair of objects. Supporting that point of view, there is a sequence of reflective embeddings $\mathbf{Pos}\hookrightarrow\Ac\hookrightarrow\Cat$.\\
There is a cofibrantly generated model structure on the category of simplicial sets, which we refer to as \emph{Quillen model structure}, generated by the sets
\begin{align*}
		I &= \left\{\partial\stdsimp{n}\to\stdsimp{n}\middle|n\in\mathbb{N}\right\}
\intertext{and}
		J &= \left\{\horn{n}{k}\to\stdsimp{n}\middle|n\in\mathbb{N}, k\le n\right\}\text{.}
\end{align*} The (barycentric) subdivision functor $\sd\colon\sSet\to\sSet$ has a right adjoint\\$\ex\colon\sSet\to\sSet$ called Kan's Ex functor, and given by $\ex(X)_n = \sSet(\sd\Delta^n,X)$. This yields an adjunction 
\begin{align*}
\adjct{\thol}{\sSet}{\Cat}{\thor}\text{,}
\end{align*}
where $N$ denotes the usual nerve functor, and $\tau_1$ its left adjoint.

In 1980, Thomason used this adjunction to lift the model structure on $\sSet$ to a Quillen equivalent model structure on $\Cat$ \cite{thomason:ccmc}, which is now called the \emph{Thomason model structure}. This allows us to lift the usual homotopy theory on simplicial sets, and thus also on topological spaces, to the category of small categories.

In 2010, Raptis showed that the Thomason model structure restricts to a model structure on $\mathbf{Pos}$, which is again Quillen equivalent to the Quillen model structure on $\sSet$, and Quillen equivalent to the Thomason model structure on $\Cat$ (cf. \cite{raptis:hotopo}). We fill the missing gap in the sequence $\mathbf{Pos}\hookrightarrow\Ac\hookrightarrow\Cat$ by showing that the Thomason model structure on $\Cat$ restricts to a model structure on $\Ac$.\\
In Section~1 we give a short introduction to acyclic categories and review a method of calculating coequalizers---and thus, in particular, pushouts---in $\Cat$ by means of generalized congruences, which were introduced by Bednarczyk, Borzyszkowski and Pawlowski in 1999 \cite{bbp:gencon}. In Section~2 we give a short introduction to model categories and cofibrantly generated model categories---mostly based on the monographs by Hovey \cite{hovey:modcat} and Hirschhorn \cite{hirschhorn:mcl}---and introduce the notion of locally finite presentability. In Section~3 we establish a model structure on $\Ac$ and show that it is Quillen equivalent to the Thomason model structure on $\Cat$.

\section{Category Theory}

In this section we introduce a few categorical notions which will be needed... In particular, we introduce the notion of Dwyer maps, which are important to understand the Thomason model structure, have a closer look at the adjunction $\adjct{p}{\Ac}{\Cat}{i}$ and develop a theory of generalized congruences. Given a category $\cat{C}$, we denote by $\cat{C}^{(0)}$ its class of objects, and by $\cat{C}^{(1)}$ its class of morphisms. Furthermore we denote by $s,t\colon\cat{C}^{(1)}\to\cat{C}^{(0)}$ the source and target map, and by $\cat{C}(x,y)$ the set of morphisms from $x$ to $y$.

\begin{mydef} Let $\cat{C}$ be a category, and $i\colon\cat{A}\to\cat{C}$ an embedding, \ie a functor that is faithful and injective on objects. We call $i$ (as well as its image in $\cat{C}$) a \emph{sieve}, if for every $y\in i(\cat{A})$, $f\in\cat{C}(x,y)$ implies $x\in i(\cat{A})$ and $f\in i(\cat{A})$. The dual of a sieve is called a \emph{cosieve}.
\end{mydef}

\begin{mydef}
	Let $i\colon\cat{A}\to\cat{C}$ be a sieve. We call $i$ a \emph{Dwyer map}, if there is a decomposition $\cat{A}\xrightarrow{f}\cat{C}'\xrightarrow{j}\cat{C}$ of $i$, such that $j$ is a cosieve in $\cat{C}$ and there is a retraction $r\colon\cat{C}'\to\cat{A}$ together with a natural transformation $\eta\colon fr\naturalto \id_{\cat{C}'}$ such that $\eta f = \id_f$.
\end{mydef}

Note that the original definition of a Dwyer map by Thomason \cite{thomason:ccmc} was stronger, in the sense that $r$ was supposed to be an adjoint to $f$. 19 years later, Cisinski introduced the weaker notion of a pseudo-Dwyer morphism in \cite{cisinski:dnpr} which proved to be more useful in our context. Since we will not need Thomsons original defintion throghout this paper, we drop the pseudo for the sake of readability.

\begin{mydef} A category $\cat{C}$ is called \emph{acyclic}\index{category!acyclic}, if it has no inverses and no nonidentity endomorphisms.
\end{mydef}

We denote by $\Ac$ the category of small acyclic categories, with morphisms the functors between acyclic categories. It is obvious that $\Ac$ is a full subcategory of $\Cat$. Hence there is a fully faithful inclusion $i\colon\Ac\to\Cat$. The inclusion $i$ has a left adjoint, called \emph{acyclic reflection}, which we construct as follows: given a category $\cat{C}\in\Cat$, define $p(\cat{C})$ to be the acyclic category with objectset
	\begin{align*}
		p(\cat{C})^{(0)} = \cat{C}^{(0)}/{\sim_o}
	\end{align*}
where ${\sim_o}$ is the equivalence relation generated by $x\sim_o y$ if $\cat{C}(x,y)\ne\emptyset\ne\cat{C}(y,x)$ and morphisms
	\begin{align*}
		p(\cat{C})^{(1)} = \cat{C}^{(1)}/{\sim_m}
	\end{align*}
	where $\sim_m$ is generated by $\id_x\sim_m\id_y$ if $x\sim_o y$ and $f\sim_m\id_x$ if $f\in\cat{C}(x,y)$ or $f\in\cat{C}(y,x)$, and $\cat{C}(x,y)\ne\emptyset\ne\cat{C}(y,x)$.

	Setting $\id_{[x]}=[\id_x]$, it is easy to see that the composition inherited from $\cat{C}$ is well defined on $p(\cat{C})$, and hence $p$ is well defined on objects. Given a functor $F\colon\cat{C}\to\cat{D}$ in $\Cat$ the components induce well defined maps on $p(\cat{C})$ via
	\begin{align*}
		p(F)([x]) = [p(F(x))]&&\text{and}&&p(F)([f]) = [p(F(f))]\text{.}
	\end{align*}
It is easy to see that this construction yields indeed a functor $p\colon\Cat\to\Ac$, which is left adjoint to the inclusion $i$. Hence $\Ac$ is reflective in $\Cat$, and we can calculate colimits in $\Ac$ by applying the acyclic reflection to the respective colimits in $\Cat$. That is:

\begin{mylem}\label{lem:colimiso} Given a diagram $D\colon I\to\Ac$, we have
	\begin{align*}
		p(\colim\nolimits_I iD) \cong (\colim\nolimits_I D)\text{.}
	\end{align*}
\end{mylem}

\begin{proof}
	This follows directly from $p$ being a right adjoint, and $pi\naturalto\id_{\Ac}$ being a natural isomorphism, since then
	\begin{equation*}
		p(\colim\nolimits_I iD)\cong (\colim\nolimits_I piD) \cong (\colim\nolimits_I D)\text{.}\qedhere
	\end{equation*}
\end{proof}

An important tool to prove that $\Ac$ inherits the Thomason model structure from $\Cat$ are generalized congruences, which were originally introduced in 1999 in \cite{bbp:gencon}, though we have chosen a notation closer to the later work by E.Haucourt, in particular \cite{haucourt:lfc}. Generalized congruences allow us to calculate coequalizers in $\Cat$, and by the previous lemma also in $\Ac$.

\begin{mydef} Given a small category $\cat{C}$, and an equivalence relation $\sim$ on the set of objects of $\cat{C}$. A \emph{$\mathord{\sim}$--composable} sequence in $\cat{C}$ is a sequence $(f_0,\dotsc,f_n)$ of morphisms in $\cat{C}$, satisfying $t(f_i)\sim s(f_{i+1})$.
\end{mydef}

\begin{mydef}
	Let $\cat{C}$ be a small category. A \emph{generalized congruence} on $\cat{C}$ is an ordered pair of equivalence relations $(\mathord{\sim}_o,\mathord{\sim}_m)$ on $\cat{C}^{(0)}$ respectively on the set of non--empty, $\mathord{\sim}_o$--composable sequences in $\cat{C}$, satisfying the following properties:
	\begin{enumerate}
		\item if $x\sim_oy$, then $(\id_x)\sim_m(\id_y)$.
		\item If $(f_0,\dotsc, f_n)\sim_m(h_0,\dotsc, h_m)$, then $t(f_n)\sim_o t(h_m)$ and $s(f_0)\sim_o s(h_0)$.
		\item If $s(h)=t(f)$, then $(f,h)\sim_m (h\circ f)$.
		\item If
		\begin{align*}
			(f_0,\dotsc, f_n)&\sim_m (f'_{0},\dotsc, f'_{n'})\text{,}\\
			(h_0,\dotsc, h_m)&\sim_m(h'_{0},\dotsc, h'_{m'})\text{, and}\\
			t(f_n)&\sim_o s(h_0)\text{,}\\
		\intertext{then }
			(f_0,\dotsc, f_n, h_0,\dotsc, h_m)&\sim_m(f'_{0},\dotsc, f'_{n'},h'_{0},\dotsc, h'_{m'})\text{.}
		\end{align*}
	\end{enumerate}

\end{mydef}

Given a generalized congruence on a category $\cat{C}$, we can define the quotient of that category, thanks to the following proposition (cf. \cite[Proposition~1.6]{haucourt:lfc}):

\begin{myprop}
	Let $(\osim_o,\osim_m)$ be a generalized congruence on a category $\cat{C}$, and $\mathbf{F}\subseteq(\cat{C}\downarrow\mathbf{Cat})$ be the full subcategory with objects being functors $F$, satisfying the following properties:
	\begin{enumerate}
		\item for all objects $x,y\in\cat{C}$, if $x\sim_o y$, then $F(x)=F(y)$, and
		\item for all $\osim_o$--composable sequences $(f_0, \dotsc, f_n)$ and $(h_0,\dotsc, h_m)$, if
			\begin{align*}
				 (f_0, \dotsc, f_n)&\sim_m (h_0,\dotsc, h_m)\text{,}\\
				\intertext{then}
				F(f_n)\circ\dotsb\circ F(f_0) &= F(h_m)\circ\dotsb\circ F(h_0)\text{.}
			\end{align*}
	\end{enumerate}
	Then $\mathbf{F}$ has an initial object, which we denote by $Q_{\osim}\colon\cat{C}\to\cat{C}/{\osim}$.
\end{myprop}

\begin{mydef}
	Given the functor $Q_{\osim}\colon\cat{C}\to\cat{C}/\osim$ as above, we call $\cat{C}/\osim$ the \emph{quotient}\index{category!quotient} of $\cat{C}$, and $Q_\osim$ the corresponding \emph{quotient functor}\index{functor!quotient}.
\end{mydef}

There is an explicit construction for the quotient category $\cat{C}/\osim$, given in \cite{bbp:gencon}: the objects of $\cat{C}/\osim$ are the equivalence classes of objects of $\cat{C}$ \wrt $\osim_o$, whereas the morphisms are given by equivalence classes of $\osim_o$--composable sequences in $\cat{C}^{(1)}$ \wrt $\osim_m$. For the sake of readability, we denote equivalence classes with respect to both relations by $[-]$.
\begin{enumerate}
	\item $(\cat{C}/\osim)^{(0)} = \left\{[x]\middle| x\in\cat{C}^{(0)}\right\}$.
	\item $(\cat{C}/\osim)^{(1)} = \left\{[(f_0,\dotsc,f_n)]\middle| f_i\in\cat{C}^{(1)}, [t(f_i)]= [s(f_{i+1})]\right\}$
	\item $\id_{[x]} = [\id_x]$
	\item $s([(f_0,\dotsc,f_n)]) = [s(f_0)]$ and $t([(f_0,\dotsc,f_n)]) = [t(f_n)]$
	\item $[(h_0,\dotsc, h_m)]\circ [(f_0,\dotsc,f_n)] = [(f_0,\dotsc ,f_n,h_0,\dotsc ,h_m)]$
\end{enumerate}

A \emph{relation} $R$ on a category $\cat{C}$ is a pair $R=(R_o, R_m)$, where $R_o$ is a relation on the set of objects of $\cat{C}$, and $R_m$ is a relation on the set of finite, nonempty sequences of morphisms of $\cat{C}$. Ordered by inclusion, they form a complete lattice. Generalized congruences are examples of relations on a category. In particular, the total relation which identifies all objects and morphisms is a generalized congruence. Hence for any relation $R$, there is a smallest generalized congruence containing $R$, which we call the \emph{principal congruence} generated by $R$. The following proposition---originally \cite[Proposition~4.1]{bbp:gencon}---allows us to construct coequalizers as quotients by principal congruences in $\Cat$:

\begin{myprop}
	Let $\cat{C}\xrightrightarrows{F}{G}\cat{D}$ be functors between small categories, let $\sim_{F=G}$ be the relation on $\cat{D}$ defined by $F(x)\sim_{F=G} G(x)$ and $F(f)\sim_{F=G} G(f)$ for all $x\in\cat{C}^{(0)}$, $f\in\cat{C}^{(1)}$. Let $\osim$ be the principal congruence on $\cat{D}$ generated by $\sim_{F=G}$. Then the quotient functor $Q_\osim\colon\cat{D}\to\cat{D}/\osim$ is the coequalizer of $F$ and $G$.
\end{myprop}

The ability to calculate coequalizer in $\Cat$ allows us, in particular, to calculate pushouts due to the following well known lemma (e.g.\ \cite[Remark~11.31]{ahs:joycat}):

\begin{mylem}\label{lem:pushcoeq}
	Let $\cat{C}$ be a category. If we have a diagram $x\xleftarrow{f}w\xrightarrow{g} y$ in $\cat{C}$, and if $x\xrightarrow{\iota_x} x\amalg y\xleftarrow{\iota_y} y$ is a coproduct and $x\amalg y\xrightarrow{h} q$ is a coequalizer of the diagram $w\xrightrightarrows{\iota_x\circ f}{\iota_y\circ g} x\amalg y$, then
	\begin{align*}
		\xymatrix{
			\ar[d]_{f}w\ar[r]^g&y\ar[d]^{h\circ\iota_y}\\
			x\ar[r]^{h\circ\iota_x}& q
			}
	\end{align*}
	is a pushout square.
\end{mylem}
\section{Model Categories}

\begin{mydef}
	Let $\cat{C}$ be a category, and $f$, $g$ be morphisms in $\cat{C}$. We say that $f$ has the \emph{left lifting property} \wrt $g$, and $g$ has the \emph{right lifting property} \wrt $f$, if given any solid arrow diagram
	\begin{align}\label{eq:rclc}\begin{aligned}
		\xymatrix
		{
			x\ar[d]_f\ar[r]&u\ar[d]^g\\
			y\ar@{-->}[ur]^h\ar[r]&v
		}\end{aligned}
	\end{align}
	there is an arrow $h\colon y\to u$, such that both triangles commute.\\
\end{mydef}

\begin{mydef}\label{def:modcat}
A \emph{model category} is a category $\cat{M}$ together with three classes of morphisms: a class of \emph{weak equivalences} $W$, a class of \emph{fibrations} $F$, and a class of \emph{cofibrations} $C$, satisfying the following properties:
	\begin{enumerate}
		\item[M1] $\cat{M}$ is bicomplete.
		\item[M2] $W$ satisfies the 2--out--of--3 property.
		\item[M3] The classes $W$, $F$ and $C$ are closed under retracts.
		\item[M4] We call a map a \emph{trivial fibration} if it is a fibration and a weak equivalence, and a \emph{trivial cofibration} if it is a cofibration and a weak equivalence. Then trivial cofibrations have the left lifting property with respect to fibrations, and cofibrations have the left lifting property with respect to trivial fibrations.
		\item[M5] Every morphism $f$ in $\cat{M}$ has two functorial factorizations:
			\begin{enumerate}
				\item $f = qi$, where $i$ is a cofibrations and $q$ is a trivial fibration.
				\item $f=pj$, where $j$ s a trivial cofibration and $j$ is a fibration.
			\end{enumerate}
	\end{enumerate}
\end{mydef}

An object $x$ in a model category is called \emph{cofibrant} if the unique morphism $\emptyset\to x$ is a cofibration, and \emph{fibrant} if the unique morphism $x\to *$ is a fibration.

\begin{myex}
	Let $\mathbf{sSet}$ the category of simplicial sets. The \emph{Quillen model structure} is given as follows: let $f\colon X\to Y$ be a simplicial map. We say $f$ is a
	\begin{enumerate}
		\item weak equivalence, if $f$ is a weak homotopy equivalence, \ie $\real{f}$ is a weak homotopy equivalence in $\mathbf{Top}$;
		\item cofibration if $f$ is a monomorphism, \ie a levelwise injection;
		\item fibration if $f$ is a Kan fibration, \ie $f$ has the right lifting property with respect to all horn inclusions.
	\end{enumerate}
	We denote this model category by $\mathbf{sSet}_{\mathrm{Quillen}}$. Note that fibrant objects \wrt the Quillen model structure are exactly Kan complexes.
\end{myex}

To work with different model categories, we need a notion of morphisms between them. These are given by \emph{Quillen adjunctions}, whereas the related notion of equivalence between model categories is given by \emph{Quillen equivalences}.

\begin{mydef}
	Let $\cat{M}$, $\cat{N}$ be model categories. A \emph{Quillen adjunction} is a pair of adjoints $\adjct{F}{M}{N}{G}$ satisfying 
	\begin{enumerate}
		\item $F$ preserves cofibrations and trivial cofibrations and
		\item $G$ preserves fibrations and trivial fibrations.
	\end{enumerate}
	A Quillen adjunction is called a \emph{Quillen equivalence} if for all cofibrant objects $x\in\cat{M}$, and all fibrant objects $y\in\cat{N}$, a map $f\colon F(x)\to y$ is a weak equivalence if and only if $\phi(f)\colon x\to G(y)$ is a weak equivalence, where $\phi\colon \Hom_{\cat{M}}(F(x),y)\xrightarrow{\cong}\Hom_{\cat{N}}(x,G(y))$ is the usual isomorphism related to the adjunction $F\dashv G$.
\end{mydef}

\begin{mydef}
	Let $I$ be a class of maps in a category $\cat{C}$. A morphism $f$ in $\cat{C}$ is
	\begin{enumerate}
		\item \emph{$I$--injective}, if it has the \emph{right lifting property} with respect to all morphisms in $I$. We denote the class of $I$--injectives by $I\inj$.
		\item \emph{$I$--projective}, if it has the \emph{left lifting property} with respect to all morphisms in $I$. We denote the class of $I$--projectives by $I\proj$.
		\item an \emph{$I$--cofibration}, if it has the \emph{left lifting property} with respect to every $I$--injective morphism. We denote the class of $I$--cofibrations by $I\cof$.
		\item an \emph{$I$--fibration}, if it has the \emph{right lifting property} with respect to every $I$--projective morphism. We denote the class of $I$--fibrations by $I\fib$.
	\end{enumerate}

\end{mydef}

A particularly useful type of model categories are proper model categories. We will only give the definition here, and refer the reader to \cite[Chapter~13]{hirschhorn:mcl} for an introductory text.

\begin{mydef}
	A model category $\cat{M}$ is called
	\begin{enumerate}
		\item \emph{left proper} if every pushout of a weak equivalence along a cofibration is a weak equivalence,
		\item \emph{right proper} if every pullback of a weak equivalence along a fibration is a weak equivalence, and
		\item \emph{proper} if it is left and right proper.
	\end{enumerate}
\end{mydef}

\subsection{Cofibrantly Generated Model Categories}
The model structures we are working with are all cofibrantly generated, and since this property is an essential ingredient 
in the following proofs, we will give a brief recap of the necessary definitions and results.

\begin{myprop}
	Let $\cat{C}$ be a cocomplete category, $\lambda$ be an ordinal. A \emph{$\lambda$--sequence} in $\cat{C}$ is a functor $X\colon\lambda\to\cat{C}$ such that for every limit ordinal $\gamma < \lambda$, the induced map
	\begin{align*}
		\colim_{\beta<\gamma} X_\beta\to X_\gamma
	\end{align*}
	is an isomorphism. The \emph{composition} of a $\lambda$--sequence is the map $X_0\to\colim_{\beta<\lambda} X_\beta$. Moreover, given a class of maps $D$ in $\cat{C}$, a \emph{transfinite composition} of maps in $D$ is the composition of a $\lambda$--sequence $X\colon\lambda\to\cat{C}$, where every morphism $X_\beta\to X_{\beta+1}$ is an element of $D$.
\end{myprop}

It is useful to note that any coproduct in a category $\cat{C}$ can be obtained as a transfinite composition of pushouts, due to the following proposition \cite[Proposition~10.2.7]{hirschhorn:mcl}:

\begin{myprop}\label{prop:dirprodlim}
	If $\cat{C}$ is a category, $S$ a set, and $f_s\colon x_s\to y_s$ a map in $\cat{C}$ for every $s$ in $S$, then the coproduct $\amalg f_s\colon\amalg x_s\to\amalg y_s$ is a transfinite composition of pushouts of the $f_s$.
\end{myprop}

Additionally, transfinite compositions allow us to introduce the notion of cell complexes:

\begin{mydef}
	Let $I$ be a class of morphisms in a cocomplete category. A \emph{relative $I$--cell complex} is a transfinite composition of pushouts along elements of $I$. An object $x\in\cat{C}$ is an \emph{$I$--cell complex}, if $0\to x$ is a relative $I$--cell complex.
\end{mydef}

$I$--cell complexes enjoy the following useful property (cf.\ \cite[Lemma~2.1.10]{hovey:modcat})

\begin{mylem}\label{lem:icellcof}
	Let $I$ be a class of morphisms in a category $\cat{C}$ with all small colimits. Then $I\mathrm{-cell}\subseteq I\cof$.
\end{mylem}

Before we are able to define cofibrantly generated model categories, we need to introduce the notion of smallness, and clarify what it means for a class of morphisms to permit the small object argument.

\begin{mydef}
	Let $\cat{C}$ be a cocomplete category, $\cat{D}\subseteq\cat{C}$. If $\kappa$ is a cardinal, then an object $x\in\cat{C}$ is \emph{$\kappa$--small relative to $\cat{D}$} if for every regular cardinal $\lambda\ge\kappa$ and every $\lambda$--sequence 
	\begin{equation*}
		X_0\to X_1\to\dotsb\to X_\beta\to\dotsb\qquad (\beta<\lambda)
	\end{equation*}
	in $\cat{C}$, such that $X_\beta\to X_{\beta+1}$ is in $\cat{D}$ for every $\beta$ with $\beta+1<\lambda$, the map of sets 
	\begin{equation*}
		\colim_{\beta<\lambda}\cat{C}(x,X_\beta)\to\cat{C}(x,\colim_{\beta<\lambda} X_\beta)
	\end{equation*}
	is an isomorphism. We say $x$ is \emph{small relative to $\cat{D}$} if it is $\kappa$--small for some ordinal $\kappa$ and we say $x$ is \emph{small} if it is small relative to $\cat{C}$.
\end{mydef}

In $\Cat$, every object is small and there is an easy way to find the an ordinal $\kappa$ such that the conditions from the previous definition are satisfied (cf.\ \cite[Proposition~7.6]{fpp:mscsdc}).

\begin{myprop}\label{prop:ksmall}
	Every category $\cat{C}\in\Cat$ is $\kappa$--small, where 
	\begin{align*}
		\kappa = |\cat{C}^{(0)}|+|\cat{C}^{(1)}|+|\cat{C}^{(1)}\btimes{s}{t}\cat{C}^{(1)}|\text{.}
	\end{align*}
\end{myprop}

In the context of cofibrantly generated model categories, it is usually enough for an object to be small \wrt to certain sets of morphisms.

\begin{mydef}
	Let $\cat{C}$ be a cocomplete category, and $I\subseteq\cat{C}^{(1)}$ be a set. An object is \emph{small relative to $I$} if it is small relative to the category of $I$--cell complexes and we say that $I$ \emph{permits the small object argument} if the domains of elements of $I$ are small relative to $I$.
\end{mydef}

Given a set of morphisms that permits the small object argument, a slightly stronger version of Lemma~\ref{lem:icellcof} holds (cf.\ \cite[Lemma~10.5.23]{hirschhorn:mcl}).

\begin{myprop}
	Let $\cat{C}$ be a cocomplete category and $I$ be a set of morphisms that permits the small object argument. Then the class of $I$--cofibrations equals the class of retracts of relative $I$--cell complexes.
\end{myprop}

We are now ready to give the definition of a cofibrantly generated model category

\begin{mydef}
	A \emph{cofibrantly generated model category} is a model category $\cat{M}$ such that:
	\begin{enumerate}
		\item There exists a set $I\subseteq\cat{M}^{(1)}$, called the set of \emph{generating cofibrations}, that permits the small object argument and satisfies $F\cap W = I\inj$.
		\item There exists a set $J\subseteq\cat{M}^{(1)}$, called the set of \emph{generating trivial cofibrations}, that permits the small object argument and satisfies $F = J\inj$.
	\end{enumerate}
\end{mydef}

The following propositions, which are \cite[Proposition~11.2.1]{hirschhorn:mcl} and \cite[Proposition~10.5.16]{hirschhorn:mcl} should give some motivation why a cofibrantly generated model category is defined the way it is.

\begin{myprop}\label{prop:cgmccf}
	Let $\cat{M}$ be a cofibrantly generated model category with generating cofibrations $I$ and generating trivial cofibrations $J$. Then:
	\begin{enumerate}
		\item The class of cofibrations of $\cat{M}$ equals the class of retracts of relative $I$--cell complexes, which equals the class of $I$--cofibrations.
		\item The class of trivial fibrations of $\cat{M}$ equals the class of $I$--injectives.
		\item The class of trivial cofibrations of $\cat{M}$ equals the class of retracts of relative $J$--cell complexes, which equals the class of $J$--cofibrations.
		\item The class of fibrations of $\cat{M}$ equals the class of $J$--injectives.
	\end{enumerate}
\end{myprop}

\begin{myprop}[The small object argument]
	Let $\cat{C}$ be a small category and $I\subseteq\cat{C}^{(1)}$. Assume that $I$ permits the small object argument. Then there is a functorial factorization of every map in $\cat{C}$ into a relative $I$--cell complex followed by an $I$--injective.
\end{myprop}

Note that given a cofibrantly generated model category, by Proposition~\ref{prop:cgmccf} the factorizations we obtain by applying the small object argument with respect to the sets $I$ and $J$ yield exactly the factorizations required in M5 of Definition~\ref{def:modcat}. Although we will not give a full proof of the small object argument, we want to give a short sketch of how the factorization works and fix some notation. Let $\cat{C}$ be a small category, and $I\subseteq\cat{C}$ be a set that permits the small object argument. Given any morphism $f\colon x\to y$ in $\cat{C}$, we obtain a factorization $x\to E_\infty\to y$ where $x\to E_\infty$ is the transfinite composition of a $\lambda$--sequence 
\begin{align*}
	x=E_0\to E_1\to\dotsb\to E_\beta\to E_{\beta+1}\to\dotsb\qquad (\beta<\lambda)\text{,}
\end{align*}
where the $E_\beta$ are obtained by pushouts of coproducts of elements of $I$.\\
We will end the discussion of cofibrantly generated model categories by giving two theorems. The first one allows us to establish a model structure on a category from the knowledge of the generating cofibrations, generating trivial cofibrations and weak equivalences (cf.\ \cite[Prop.~11.3.1]{hirschhorn:mcl}). The second is commonly known as \emph{Kan's Lemma on Transfer} and allows us to transport a model structure along an adjunction (cf. \cite[Theorem~11.3.2]{hirschhorn:mcl}).

\begin{myprop}\label{thm:mswij}
	Let $\cat{C}$ be a bicomplete category. Suppose $\cat{W}$ is a subcategory and that $I,J\subseteq\cat{C}^{(1)}$ are sets. Then $\cat{C}$ is a cofibrantly generated model category with generating cofibrations $I$ and generating trivial cofibrations $J$ and subcategory of weak equivalences $\cat{W}$, if and only if the following conditions hold:
	\begin{enumerate}
		\item $\cat{W}$ satisfies the 2--out--of--3 property and is closed under retracts.
		\item The domains of $I$ are small relative to $I$--cell.
		\item The domains of $J$ are small relative to $J$--cell.
		\item $J\mathrm{-cell}\subseteq\cat{W}\cap I\mathrm{-cof}$.
		\item $I\mathrm{-inj}\subseteq\cat{W}\cap J\mathrm{-inj}$.
		\item $\cat{W}\cap I\mathrm{-cof}\subseteq J\mathrm{-cof}$ or $\cat{W}\cap J\mathrm{-inj}\subseteq I\mathrm{-inj}$.
	\end{enumerate}
\end{myprop}

The important example of a cofibrantly generated model category in the context of this paper is the Thomason model structure on $\Cat$. Let $\sd\colon\sSet\to\sSet$ be the (barycentric) subdivision functor, and $\operatorname{Ex}\colon\sSet\to\sSet$ be its right adjoint (cf. \cite{kan:oncsscomp}).\\
Let furthermore $\nerve\colon\Cat\to\sSet$ denote the nerve functor, and $\tau_1\colon\sSet\to\Cat$ its right adjoint. Then we have an adjunction $\thol\colon\sSet\leftrightarrows\Cat:\!\thor$. Define
\begin{align*}
	W& := \left\{f\in\Cat^{(1)}\middle|\nerve(f)\text{ is a weak equivalence}\right\}\text{,}\\
	I& := \left\{\thol\partial\stdsimp{n}\to\thol\stdsimp{n}\middle| n\ge0\right\}\text{,}\\
	J& := \left\{\thol\horn{n}{k}\to\thol\stdsimp{n}\middle| n\ge0, n\ge k\ge 0\right\}\text{.}
\end{align*}
If we denote by $\cat{W}$ the wide subcategory of $\sSet$ satisfying $\cat{W}^{(0)}=W$, then $\cat{W}$, $I$ and $J$ satisfy the conditions of Theorem~\ref{thm:mswij} and define a models structure on $\Cat$, which is known as the \emph{Thomason model structure} and has the property that the adjunction $\thol\colon\sSet\leftrightarrows\Cat:\!\thor$ is a Quillen equivalence with respect to the Quillen model structure on $\sSet$.\\

\begin{myprop}[Kan's Lemma on Transfer]\label{prop:liftmodcat}
	Let $\cat{M}$ be a cofibrantly generated model category with generating cofibrations $I$ and generating trivial cofibrations $J$. Let $\cat{N}$ be a category that is closed under small limits and colimits and let $\adjct{F}{M}{N}{U}$ be a pair of adjoint functors. Define $FI:=\{Fu|u\in I\}$ and $FJ:=\{Fv|v\in J\}$. If
	\begin{enumerate}
		\item $FI$ and $FJ$ permit the small object argument and
		\item $U$ takes relative $FJ$--cell complexes to weak equivalences,
	\end{enumerate}
	then there is a cofibrantly generated model category structure on $\cat{N}$ where $FI$ is a set of generating cofibrations, $FJ$ is a set of generating trivial cofibrations, and the weak equivalences are the maps that $U$ takes into weak equivalences in $\cat{M}$.
	Furthermore, with respect to this model structure, $F\dashv U$ is a Quillen adjunction.
\end{myprop}

\subsection{Locally Presentable Categories}
As pointed out in \cite[Remark~1.2]{beke:shmc}, locally presentable categories enjoy the property that every set of morphisms permits the small object argument. It is a well known fact that $\Cat$ is locally (finitely) presentable, and we will show later that $\Ac$ inherits that property. We will keep this short and refer the interested reader to \cite{ar:lopac}.\\
Recall that a poset $(P,\le)$ is called \emph{directed} if every pair of elements has an upper bound, and that a colimit of a diagram $X\colon I\to\cat{C}$ is called \emph{directed colimit} if $I$ is a directed poset.

\begin{mydef}
	An object $x$ of a category $\cat{C}$ is called \emph{locally finitely presentable} if the homfunctor
	\begin{equation*}
		\operatorname{hom}(x,-)\colon\cat{C}\to\mathbf{Set}
	\end{equation*}
	preserves directed colimits.
\end{mydef}

\begin{mydef}
	 A category $\cat{C}$ is called \emph{locally finitely presentable} if it is cocomplete and has a set $A$ of finitely presentable objects such that every object of $\cat{C}$ is a directed colimit of objects of $A$.
\end{mydef}

There is a useful theorem (cf.~\cite[Theorem~1.39]{ar:lopac}) that allows us to decide whether a reflective subcategory of a locally presentable category is locally presentable:

\begin{mylem}\label{lem:dircolylfp}
	Let $\cat{C}$ be a locally $\lambda$--presentable category and $\cat{A}\subseteq\cat{C}$. If $\cat{A}$ is reflective and the inclusion $i\colon\cat{A}\to\cat{C}$ preserves $\lambda$--directed colimits, then $\cat{A}$ is locally $\lambda$--presentable.
\end{mylem}

Recall the definition of a filtered colimit:

\begin{mydef}\label{def:filtered} A non-empty category $\cat{C}$ is called \emph{filtered}\index{category!filtered}, if
	\begin{enumerate}
		\item for every pair of objects $x_1$, $x_2$ in $\cat{C}$ there is an object $y$ in $\cat{C}$ and morphisms $f_i\colon x_i\to y$, $i=1,2$, and
		\item for any pair of parallel morphisms $f_1, f_2\colon x\to y$ there exists an object $z$ and a morphism $h\colon y\to z$, such that $hf_1= hf_2$.
	\end{enumerate}
	We call a diagram $D\colon I\to\cat{C}$ filtered if the index category $I$ is filtered and a colimit is called \emph{filtered colimit} if it is a colimit over a filtered diagram.
\end{mydef}

Sometimes it is easier to check whether a functor preserves filtered colimits instead of directed, and the following lemma (cf.~\cite[p.~15]{ar:lopac}) allows us to do so:

\begin{mylem}\label{lem:fildir}
	A functor $F\colon\cat{C}\to\cat{D}$ preserves filtered colimits if and only if it preserves directed colimits.
\end{mylem}

In particular in $\Cat$, there is an explicit method to construct filtered colimits, which can---for example---be found as \cite[Proposition~2.13.3]{borceux:hca1} and \cite[5.2.2f]{borceux:hca2}:

\begin{myprop}\label{prop:colimfilset}
	Let $D\colon I\to\mathbf{Set}$ be a filtered diagram, then 
	\begin{align*}
		\colim_I D = (C,s_i\colon X_i\to C)_{i\in I}\text{,}
	\end{align*}
	where the set $C$ is given by
	\begin{equation*}
		C = \coprod_{i\in I} X_i/\mathord{\sim}\text{,}
	\end{equation*}
	where $\sim$ is defined as follows: given $x\in X_i$, $x'\in X_{i'}$, we have $x\sim x'$ if there exists a $j\in I$, together with maps $f\colon X_i\to X_j$ and $g\colon X_{i'}\to X_j$ such that $f(x) = g(x')$, and the maps $s_i$ are given as
	\begin{align*}
		s_i \colon X_i&\longrightarrow C\text{,} \\
		x&\longmapsto [x]\text{.}
	\end{align*}
\end{myprop}

Note that given a filtered diagram $D\colon I\to\mathbf{Set}$, if $x\in D_i$, and $[x]\in\colim_I D$, then given any morphism $D(i\to j)\colon D_i\to D_j$, we have $[x] = [D(i\to j)(x)]$.

\begin{myprop}\label{prop:colimfilcat}
	Let $D\colon I\to\mathbf{Cat}$ be a filtered diagram. There is an explicit description of $\cat{L}=\colim_I D$ given as follows: $\cat{L}^{(0)}=\colim_I D(i)^{(0)}$ is just the usual colimit in $\mathbf{Set}$. 
	Given a pair of objects $L$, $L'$ in $\cat{L}^{(0)}$, the morphism set $\cat{L}(L,L')$ is given by the colimit $\colim_I D_i(L_i,L_i')$ in $\mathbf{Set}$, where $L=[L_i]$ and $L'=[L_i']$.
\end{myprop}

\section{A Model structure on $\Ac$}

In this section we will establish a model structure on the category $\Ac$. For that purpose, we will show that the inclusion $i\colon\Ac\to\Cat$ preserves filtered colimits, and that pushouts of acyclic categories along sieves are again acyclic categories. We will use these features to show that we can lift the Thomason model structure on $\Cat$ along the adjunction $p\dashv i$ and obtain a model structure on $\Ac$.

\begin{myprop}\label{prop:ipfc}
	The inclusion $i\colon\Ac\to\Cat$ preserves filtered colimits.
\end{myprop}

\begin{proof}
	Let $D\colon I\to\mathbf{Cat}$ be a filtered diagram such that $D_i$ is an acyclic category for every $i$ in $I$, and let $\cat{C}=\colim_I D$. At first, we will prove that any endomorphisms in $\cat{C}$ is necessarily the identity and secondly, we will show that now there are antiparallel morphisms in $\cat{C}$.\\
	To prove that any endomorphism is an identity, assume that there is an $x\in\cat{C}$, and an $[f]\in\cat{C}(x,x)$, such that $[f]\ne\id$. Hence, there is a category $D_i$, with objects $x_i, x_i'\in D_i$, such that $f\in D_i(x_i,x_i')$ and $x_i, x_i'\in x$. From the description of filtered colimits in $\Cat$, we know that there is a category $D_j$ and functors $F\colon D_i\to D_j$, $G\colon D_i\to D_j$ such that $F(x_i) = G(x_i')$. Since $D$ is filtered, there is a category $D_k$ and a functor $H\colon D_j\to D_k$ such that $H\circ F = H\circ G$. But $D_k$ is acyclic, and thus $H\circ F(f)=H\circ G(f) = \id_{H\circ F(x_i)}$. By Prop.\ \ref{prop:colimfilcat} this yields $[f]=[\id_{H\circ F(x_i)}] = \id_x$.\\
	We now want to show that there are no antiparallel morphisms in $\cat{C}$. Therefore we assume that there are objects $x,y\in\cat{C}$, together with two morphisms $[f]\colon x\to y$ and $[h]\colon y\to x$. By the construction of filtered colimits in $\Cat$ there are categories $D_i$ and $D_{i'}$ such that $f\in D_i(x_i,y_i)$, $h\in D_{i'}(y_{i'},x_{i'})$ and $[x_i]=[x_{i'}] = x$ as well as $[y_i]=[y_{i'}] = y$. We will use filteredness of $I$ and the construction of filtered colimits in $\Cat$ to construct the following diagram in five consecutive steps:
	
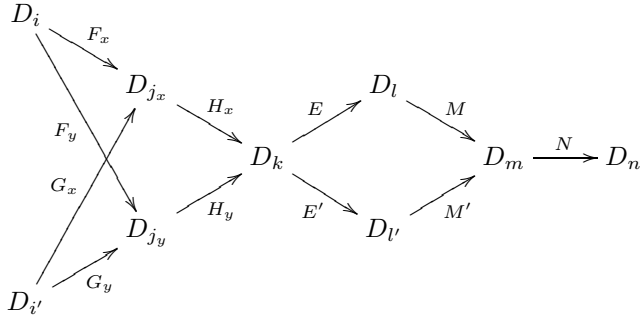
\begin{align*}
		\xymatrix@R-1pc
		{
		D_i\ar[dddr]_{F_y}\ar[dr]^{F_x}&&&&&\\
		&D_{j_x}\ar[dr]^{H_x}&&D_l\ar[dr]^{M}&&\\
		&&D_k\ar[ur]^{E}\ar[dr]_{E'}&&D_m\ar[r]^N&D_n\\
		&D_{j_y}\ar[ur]_{H_y}&&D_{l'}\ar[ur]_{M'}&&\\
		D_{i'}\ar[uuur]^{G_x}\ar[ur]_{G_y}
		}
	\end{align*}
	First, by Prop.\ \ref{prop:colimfilcat}, there are categories $D_{j_x}$ and $D_{j_y}$, together with pairs of functors $F_x\colon D_i\to D_{j_x}$, $G_x\colon D_i'\to D_{j_x}$ and $F_y\colon D_i\to D_{j_y}$, $G_y\colon D_i'\to D_{j_y}$ satisfying $F_x(x_i) = G_x(x_{i'})$ and $F_y(y_i) = G_y(y_{i'})$. Using Def.\ \ref{def:filtered} (i), there is a category $D_k$ together with functors $H_x\colon D_{j_x}\to D_k$, $H_y\colon D_{j_y}\to D_k$.
	In particular, we have 
	$$H_x\circ F_x\ne H_y\circ F_y\colon D_i\rightrightarrows D_k$$
	and 
	$$H_x\circ G_x\ne H_y\circ G_y\colon D_{i'}\rightrightarrows D_k\text{.}$$
	Thus, by Def.\ \ref{def:filtered} (ii), there are categories $D_l$ and $D_{l'}$, together with functors $E\colon D_k\to D_l$ and $E'\colon D_{k'}\to D_{l'}$ satisfying 
	$$E\circ H_x\circ F_x = E\circ H_y\circ F_y$$
	and 
	$$E'\circ H_x\circ G_x = E'\circ H_y\circ G_y\text{.}$$
	Again by Def.\ \ref{def:filtered} (i), there is a category $D_m$ and functors $M\colon D_{l}\to D_{m}$, $M'\colon D_{l'}\to D_{m}$. Yet again by Def.\ \ref{def:filtered} (ii) there is a category $D_n$ and a functor $N\colon D_m\to D_n$ satisfying 
	$$N\circ M\circ E = N \circ M' \circ E'\text{.}$$
	Putting together the previous equations, we have 
	\begin{align*}
		&N\circ M\circ E\circ H_x\circ F_x(x_i) \\
		=\,& N\circ M'\circ E'\circ H_y\circ G_y(x_{i'}) =: x_n
	\end{align*}
	and 
	\begin{align*}
		&N\circ M\circ E\circ H_x\circ F_x(y_i) \\
		=\,& N\circ M'\circ E'\circ H_y\circ G_y(y_{i'}) =: y_n\text{.}
	\end{align*}
	Hence 
	\begin{align*}
		N\circ M\circ E\circ H_x\circ F_x(f)\in D_n(x_n, y_n)
	\end{align*}
	and 
	\begin{align*}
		 N\circ M'\circ E'\circ H_y\circ G_y(h)\in D_n(y_n, x_n)\text{,}
	\end{align*}
	which contradicts that $D_n$ is an acyclic category. Thus, the subcategory of acyclic categories is closed under taking filtered colimits, which yields in particular, that the inclusion $i\colon\Ac\to\Cat$ commutes with filtered colimits.
\end{proof}

Lemma~\ref{lem:dircolylfp} in conjunction with Lemma~\ref{lem:fildir} yields immediately:

\begin{mycor}
	The category $\Ac$ is locally finitely presentable.
\end{mycor}

The next step is to prove that pushouts of acyclic categories along sieves in $\Cat$ are again acyclic categories. For that purpose we need a few preparational lemmas. The first of which can be found in \cite[Proposition~5.2]{frla:hin}, the second we will prove here.

\begin{mylem} \label{lem:umsiepush}Given a pushout
	\begin{equation}
		\begin{gathered}
		\xymatrix{
			\cat{A}\ar[r]^F\ar[d]^i&\cat{C}\ar[d]^j\\
			\cat{B}\ar[r]^G&\cat{B}\amalg_{\cat{A}}\cat{C}
		}
		\end{gathered}
		\label{diag:pushsieve}
	\end{equation}
 where $i\colon\cat{A}\to\cat{B}$ is a sieve, then $j$ is a full inclusion, i.e. bijective on objects and morphisms.
\end{mylem}

\begin{mylem}\label{lem:elsiepush}
	Given the pushout diagram~\eqref{diag:pushsieve}, every element $[x]\in\cat{B}\amalg_{\cat{A}}\cat{C}$ satisfies either
	\begin{enumerate}
		\item $[x]=\{x\}$ and $x\in\cat{B}^{(0)}\setminus i(\cat{A}^{(0)})$, or
		\item there is one and only one $c\in\cat{C}$, such that $[x]=[c]$.
	\end{enumerate}
\end{mylem}

\begin{proof}
	Assumption (i) is obvious, since $x$ has no preimage in $\cat{A}$, it is only equivalent to itself. On the other hand, if $x$ is not in $\cat{B}^{(0)}\setminus i(\cat{A}^{(0)})$ it has a preimage in $\cat{A}$, which has an image in $\cat{C}$ and then $(ii)$ follows directly from Lemma~\ref{lem:umsiepush}.
\end{proof}

\begin{myprop}\label{prop:pushsieve}
	Let $\cat{B}\xleftarrow{i}\cat{A}\xrightarrow{F}\cat{C}$ be a diagram of acyclic categories, and assume that $i$ is a sieve. Then the pushout in $\mathbf{Cat}$ is again an acyclic category.
\end{myprop}

\begin{proof}
	The pushout of the given diagram is given by the coequalizer $Q$ of the diagram $\cat{A}\xrightrightarrows{\iota_{\cat{B}}\circ i}{\iota_{\cat{C}}\circ F}\cat{B}\amalg\cat{C}$. Where $Q$ is the quotient of $\cat{B}\amalg\cat{C}$ by the principal general congruence $(\osim_o,\osim_m)$ generated by the relation $\sim_{\iota_{\cat{C}}\circ i=\iota_{\cat{D}}\circ F}$. For the sake of convenience, we will subsequently ignore the inclusions $\iota_\cat{B}$ and $\iota_\cat{C}$ from notation, and simply write $f\in\cat{B}$ for a morphism $f$ in the image $\iota_\cat{B}(\cat{B})$.\\
	By Lemma~\ref{lem:elsiepush}, $Q^{(0)}\cong\left(\cat{B}^{(0)}\setminus i\left(\cat{A}^{(0)}\right)\right)\amalg \cat{C}^{(0)}$. Hence a morphism $f=[(f_0,\dotsc, f_n)]$ in $Q$ satisfies either 
	\begin{enumerate}
		\item $f_0,\dotsc ,f_n\in\cat{B}\setminus i\left(\cat{A}^{(0)}\right)$,
		\item either $f_i\in i(\cat{A})$, or $f_i\in\cat{C}$ for every $i=0,\dotsc, n$, or
		\item there is a $0\le k\le n$, such that:
		\begin{align*}
			f_i\in i(\cat{A})\text{ or } f_i\in\cat{C}&\qquad \text{for }i<k\\
			s(f_k)\in\cat{C}^{(0)}\amalg\cat{A}^{(0)} \text{ and } t(f_k)\in\cat{B}^{(0)}&\\
			f_i\in\cat{B}\setminus\cat{A}^{(0)}&\qquad \text{for }i>k
		\end{align*}
	\end{enumerate}
	In case (i), $(f_0,\dotsc, f_n)\sim_m f_n\circ\dotsb\circ f_0$, since $\cat{B}\setminus i\left(\cat{A}^{(0)}\right)$ embeds fully into $Q$. Thus, in particular, $[t(f_n)]\ne [s(f_0)]$ and $Q([t(f_n)],[s(f_0)])=\emptyset$.\\
	Considering case (ii), we claim that there is a composable sequence of morphisms $(h_0,\dotsc,h_n)$ in $\cat{C}$, such that $(f_0,\dotsc, f_n)\sim_m(h_0,\dotsc,h_n)$. Note therefore, that given any $\osim_o$--composable pair of morphisms $f_i$, $f_{i+1}$ in $\cat{B}\amalg\cat{C}$, satisfying condition (ii), we have $t(f_{i})\sim_os(f_{i+1})$. Hence by Lemma~\ref{lem:elsiepush}, there is a unique $x\in\cat{C}^{(0)}$, such that $x\sim_ot(f_{i})\sim_os(f_{i+1})$. Moreover, since $f_i$, $f_{i+1}$ have preimages in $\cat{A}$, by Lemma~\ref{lem:umsiepush} there are unique morphisms $h_i = F(i^{-1}(f_i))$, $h_{i+1}= F(i^{-1}(f_{i+1}))$, such that $t(h_i)\sim_o x$, and $t(h_{i+1})\sim_o t(f_{i+1})$, and since $x\sim_ot(h_i)\sim_os(h_{i+1})$, and $x$ is unique, $h_i$ and $h_{i+1}$ are composable. Thus there is a composable sequence $(h_0,\dotsc,h_n)$ of morphisms in $\cat{C}$, such that $(f_0,\dotsc, f_n)\sim_m(h_0,\dotsc,h_n)$. By definition of a generalized congruence, $(h_0,\dotsc,h_n)\sim_m h_n\circ\dotsb\circ h_0=: h$. Since $h$ is a morphism in $\cat{C}$, and $\cat{C}$ embeds fully into $Q$ by Lemma~\ref{lem:umsiepush}, it follows that $s([h])\ne t([h])$. Furthermore, by the same argument a morphism $[(f_0',\dotsc,f_n')]$ in $Q(t([h]),s([h]))$ would yield a morphism $h'\in\cat{C}(t(h),s(h))$, which contradicts $\cat{C}$ being acyclic.\\
	In case (iii), if $k=0$, $(f_0,\dotsc,f_n)\sim_m f_n\circ\dotsc\circ f_0=:f$, since $f_k$ has no preimage in $\cat{A}$ for every $k=0,\dotsc,n$, hence $[f_k]=\{f_k\}$ and thus composition is well defined. Moreover, $s(f)\ne t(f)$ by construction. And $Q(t(f), s(f))=\emptyset$ since $i(\cat{A})$ is a sieve.\\
	If $k\ne 0$, we can decompose $[(f_0,\dotsc, f_n)]$ into $[(f_k,\dotsc, f_n)]\circ[(f_0,\dotsc, f_{k-1})]$, apply the former arguments to the individual morphisms and use the fact that $s(f_0)\ne t(f_n)$ by construction.
\end{proof}

\begin{mythm}
	Consider the morphism sets
	\begin{align*}
		I = \left\{\thol\partial\stdsimp{n}\to\thol\stdsimp{n}\middle|n\in\mathbb{N}\right\}
\intertext{and}
		J = \left\{\thol\horn{n}{k}\to\thol\stdsimp{n}\middle|n\in\mathbb{N}, k\le n\right\}\text{.}
	\end{align*}
	in $\Cat$ and the adjunction $p\colon\Cat\leftrightarrows\Ac:\!i$. $\Ac$ is a proper combinatorial cofibrantly generated model category with generating cofibrations $pI$ and generating trivial cofibrations $pJ$, $p\dashv i$ is a Quillen equivalence.
\end{mythm}

\begin{proof}
	Remember that the sets $I$ and $J$ are the generating cofibrations and generating trivial cofibrations for the Thomason model structure on $\Cat$. By \cite[Lemma~5.1]{thomason:ccmc}, the domains and codomains of $I$ and $J$ are posets, and by Proposition~\ref{prop:ksmall} $\kappa$--small for some finite ordinal $\kappa$. Moreover, by \cite[Proposition~4.2]{thomason:ccmc} elements of $I$ and $J$ are Dwyer morphisms. Let $f\colon x\to y$ be a morphism in $\Ac$. Since $\Cat$ is a cofibrantly generated model category, the small object argument yields a factorization $i(x)\xrightarrow{j'}E'_\infty\xrightarrow{q'} i(y)$ of $i(f)$ in $\Cat$. We know that $\kappa$ is finite, that $i$ preserves filtered colimits (and by Lemma~\ref{lem:fildir} also directed colimits) and pushouts along sieves, and that coproducts can be expressed as $\lambda$--composable sequences. Thus applying the small object argument to $f$ in $\Ac$ \wrt to $pI$ or $pJ$ yields a factorization $x\xrightarrow{j}E_\infty\xrightarrow{q} y$ satisfying $i(j)\cong j'$, $i(E_\infty)\cong E'_\infty$, and $i(q)= q'$. Hence, factorizations of morphisms between acyclic categories in $\Cat$ are identical to the inclusions of the factorizations of the respective morphisms in $\Ac$. In particular, the sets $pI$ and $pJ$ permit the small object argument and satisfy condition (i) of Proposition~\ref{prop:liftmodcat}.\\
	Furthermore, since $\Cat$ is a cofibrantly generated model category, by Lemma~\ref{lem:icellcof} and Proposition~\ref{prop:cgmccf} (iii) every relative $J$--cell complex is a trivial cofibration in $\Cat$. Since analogously to the previous reasoning, $i$ maps $pJ$--cell complexes to \mbox{$J$--cell} complexes in $\Cat$, condition (ii) of Proposition~\ref{prop:liftmodcat} is satisfied. Thus $pI$ and $pJ$ are generating cofibrations and generating trivial cofibrations for a cofibrantly generated model structure on $\Ac$ and the adjunction $p\dashv i$ is a Quillen adjunction.\\
	The category $\Ac$ is left proper, because every cofibration is a Dwyer morphism by Proposition~\ref{prop:cgmccf} (i) and \cite[Proposition~2.4 (a)]{raptis:hotopo}, and pushouts along Dwyer morphisms in $\Ac$ are the same as in $\Cat$ by Proposition~\ref{prop:pushsieve}. The category $\Ac$ is right proper, because $\Cat$ is right proper and $i$ is a right adjoint, thus preserves pullbacks.\\
	To show that $p\dashv i$ is a Quillen equivalence, note that by \cite[Proposition~5.7]{thomason:ccmc}, every cofibrant object $\cat{C}$ in $\Cat$ is a poset, thus (in particular) an acyclic category. Hence the unit component $\eta_\cat{C}\colon \cat{C}\to ip(\cat{C})$ is an isomorphism. Let $\phi\colon\Ac(p(\cat{C}),\cat{D})\to\Cat(\cat{C},i(\cat{D}))$ denote the natural isomorphism related to $p\dashv i$. Given $f\colon p(\cat{C})\to\cat{D}$ in $\Ac$, we have $\phi(f)=i(f)\circ\eta_\cat{C}$. Since $W$ is closed under isomorphism, $\phi(f)$ is a weak equivalence if and only if $i(f)$ is, and by Proposition~\ref{prop:liftmodcat} $i(f)$ is a weak equivalence if and only if $f$ is. Thus $p\dashv i$ is a Quillen equivalence.
\end{proof}

\bibliography{references}{}
\bibliographystyle{amsalpha}

\end{document}